\documentclass[twoside,letterpaper,12pt,reqno]{article}
\usepackage[portrait,margin=1in]{geometry}
\usepackage{graphicx,color,fancyhdr}
\usepackage{amsmath,amssymb,amsthm,amsfonts,amsbsy,latexsym,dsfont}
\usepackage{mathrsfs}
\usepackage{enumitem}
\usepackage[
    colorlinks=true,%
    breaklinks,
    linkcolor=blue,%
    citecolor=blue,%
    urlcolor=blue
  ]{hyperref} 

\newtheorem{thm}{Theorem}
\newtheorem{cl}{Claim}[section]

\newtheorem{lem}[thm]{Lemma}

\newtheorem*{conj}{Conjecture}

\def\qed{\hfill \ifhmode\unskip\nobreak\fi\quad\ifmmode\Box\else$\Box$\fi\\ }
\def\COMMENT#1{}
\let\COMMENT=\footnote
\def\BB{{\mathcal B}}
\def\AA{{\mathcal A}}
\def\FF{{\mathcal F}}
\def\HH{{\mathcal H}}
\def\UU{{\mathcal U}}

\newcommand{\mycomment}[1]{}

\title{Equitable coloring of 
 planar graphs with maximum degree at least eight}

\author{
{{Alexandr Kostochka}}\thanks{
\footnotesize {University of Illinois at Urbana--Champaign, Urbana, IL 61801. E-mail: \texttt {kostochk@math.illinois.edu}.
 Rese arch 
is supported in part by  NSF  Grant DMS-2153507 and by NSF RTG Grant DMS-1937241.
}}
\and
{{Duo Lin}}\thanks{
\footnotesize {University of Illinois at Urbana--Champaign, Urbana, IL 61801. E-mail: \texttt {duolin2@illinois.edu}.
}}
\and{{Zimu Xiang}}\thanks{University of Illinois at Urbana--Champaign, Urbana, IL 61801. E-mail: {\tt zimux2@illinois.edu}. }}

\date{November 10, 2023}

\begin{document}
\maketitle

\begin{abstract}
The Chen-Lih-Wu Conjecture states that each connected graph with maximum degree $\Delta\geq 3$ that is not the complete graph
$K_{\Delta+1}$ or the complete bipartite graph $K_{\Delta,\Delta}$   admits an equitable coloring with $\Delta$ colors. 
For planar graphs, the conjecture has been confirmed for $\Delta\ge 13$ by Yap and Zhang and for $9\leq \Delta\le 12$ by Nakprasit. 
In this paper, we present a proof that confirms the conjecture for graphs embeddable into a surface with non-negative Euler characteristic with maximum degree $\Delta\geq 9$ and for planar graphs with maximum degree $\Delta\geq 8$.
\end{abstract}



\section{Introduction}
For a graph $G$, $\Delta(G)$ denotes the maximum degree of $G$.
An \emph{equitable coloring} of a graph is a proper vertex coloring such that for any two color classes $V_i$ and $V_j$, we have that $||V_i|-|V_j||\le 1$. 
A graph $G$ is \emph{equitably $k$-colorable} if it has an equitable coloring with $k$ colors.

The Hajnal-Szemer\'edi Theorem~\cite{HS} states that every graph $G$ is equitably $k$-colorable for any $k\ge\Delta(G) +1$. 
The bound is sharp for complete graphs $K_{\Delta+1}$ and for complete bipartite graphs $K_{\Delta,\Delta}$ when $\Delta$ is odd. 
 Chen, Lih and Wu~\cite{CLW} conjectured the following strengthening of the Hajnal-Szemer\'edi Theorem.

\begin{conj}[\textbf{Chen-Lih-Wu Conjecture~\cite{CLW}}]
If $G$ is an $r$-colorable graph with $\Delta(G)\leq r$, then either $G$ has an equitable $r$-coloring, or $r$ is odd and $K_{r,r}\subseteq G$.
\end{conj}

Lih and Wu~\cite{lih} proved the conjecture for bipartite graphs. Chen, Lih and Wu~\cite{CLW} themselves proved the conjecture for $r=3$ and for $r\geq |V(G)|/2$. 
Kierstead and Kostochka proved the conjecture in~\cite{KK4} for $r=4$ and  in \cite{KK} for $r>|V(G)|/4$. 
Yap and Zhang~\cite{YZ} proved that the conjecture holds for planar graphs when $r\geq 13$ and 
 Nakprasit~\cite{KN,N9} confirmed the conjecture for planar graphs when $9\leq r\leq 12$. These two results together can be stated
as follows.
\begin{thm}[Yap and Zhang~\cite{YZ}  and Nakprasit~\cite{KN,N9}]\label{YZN}
If $r\geq 9$ and
$G$ is a planar graph with $\Delta(G)\leq r$, then  $G$ has an equitable $r$-coloring.
\end{thm}

 Zhang~\cite{1planar} proved the conjecture for the wider class of $1$-planar graphs (and more generally, for graphs with maximum average degree less than $8$) but with the stronger restriction on $r$: for  $r\geq 17$.


For lower maximum degrees, Chen-Lih-Wu Conjecture was proved for planar graphs with extra restrictions, mainly with restrictions on cycle structure. 
In 2008, Zhu and Bu~\cite{zhu2008equitable} proved that the conjecture holds for $C_3$-free planar graph with maximum degree $\Delta\geq 8$. It also holds for $C_4,C_5$-free planar graphs with maximum degree $\Delta\geq 7$. In 2009, Li and Bu~\cite{li2009equitable} proved that the conjecture holds for $C_4,C_6$-free planar graph with maximum degree $\Delta\geq 6$. In 2012, Nakprasit and Nakprasit~\cite{kknak2012equitable} proved that the conjecture holds for $C_3$-free planar graphs with maximum degree $\Delta\geq 6$, $C_4$-free planar graphs with maximum degree $\Delta\geq 7$, and planar graphs with maximum degree $
\Delta\geq 5$ and girth at least $6$.

The aim of this paper is twofold. First, we present a significantly shorter proof of Theorem~\ref{YZN}. In fact, we prove it for a slightly broader class of graphs
 embeddable into a surface with non-negative Euler characteristic. For simplicity, we call such graphs \emph{semi-planar}. 

\begin{thm}\label{clww}
    If $r\geq 9$ and
$G$ is a semi-planar graph with $\Delta(G)\leq r$, then  $G$ has an equitable $r$-coloring.
\end{thm}

Our second goal is to extend Theorem~\ref{YZN} to planar graphs with maximum degree $8$:

\begin{thm}\label{pla8}
  If $r\geq 8$ and
$G$ is a planar graph with $\Delta(G)\leq r$, then  $G$ has an equitable $r$-coloring.
\end{thm}


The structure of the paper is as follows. In the next section we introduce notation, cite a known lemma and set up the proofs of both theorems. In Section~3 we prove the easier Theorem~\ref{clww}, and in the longer Section~3 we prove  Theorem~\ref{pla8}.

\section{Preliminaries and setup of proofs}

Most notation used in the paper is standard. 
For a graph $G$, let $\Delta(G)$ denote the maximum degree of $G$,  $\delta(G)$ denote the minimum degree of $G$ and  $\delta^*(G)$ denote the minimum degree over non-isolated vertices in $G$.
For a vertex subset $V\subseteq V(G)$ and some vertex $x\in V$ and $u\not\in V$, we use $V-x$ to denote $V\setminus \{x\}$ and $V+u$ to denote $V\cup\{u\}$. 
For an edge $xy\in E(G)$, $G-xy$ denotes the graph obtained by removing $xy$ from $G$. 
For two vertex subsets $X,Y\subseteq V(G)$, we use $E_G(X,Y)$ to denote the set of edges connecting $X$ with $Y$. 

For a graph $G$, $|G|$ denotes the number of vertices of $G$ and $||G||$ denotes the number of edges of $G$.

 Euler's Formula yields the following simple claim.
\begin{lem}\label{pla}
 (a)   For each  planar graph $G$ with $n\geq 3$ vertices,  $\|G\|\leq 3n-6$ and $\delta(G)\leq 5$.  For each  semi-planar graph $G$ with $n\geq 3$ vertices,  $\|G\|\leq 3n$ and $\delta(G)\leq 6$.

(b)    For each bipartite planar graph $G$ with $n\geq 3$ vertices,  $\|G\|\leq 2n-4$ and $\delta(G)\leq 3$.  
 For each bipartite semi-planar graph $G$ with $n\geq 3$ vertices,  $\|G\|\leq 2n$ and $\delta(G)\leq 4$.  
\end{lem}

We now show that it is sufficient to only consider graphs of order $rs$ for some integer $s$. 

\begin{lem}\label{equ}
    It is enough to prove Theorems~\ref{clww} and~\ref{pla8} for graphs $F$ with $|F|$ divisible by $r$.
\end{lem}

{\bf Proof.}
Suppose the theorem holds for graphs $F$ with $|F|$ divisible by $r$. Let $G$ be a semi-planar (or planar) graph with
$|G|=n=rs-p$, where $0<p<r$. If $1\leq p\leq 4,$
 then set $G'=G+K_{p}$. In this case, $G'$ remains semi-planar (or planar).
 By construction, $|G'|=n+p$ is divisible by $r$ and $\Delta(G')\leq r$. 
 So  $G'$ has an equitable $r$-coloring $f'$. All vertices of the added $K_p$ have different colors in $f'$, and hence
  the restriction of $f'$ to $G$ is an equitable $r$-coloring of $G$. 
  
Suppose now $p\geq 5$.  
By Lemma~\ref{pla}(a), either $G$ is $6$-regular or $G$ has a vertex
$v_1$ of degree at most $5$. In the first case, the theorem follows from the Hajnal-Szemer\' edi Theorem. In the second case,
 we can order the vertices of $G$ as $v_1,\ldots,v_n$ so that for each $2\leq i<n$, $d_{G-\{v_1,\ldots,v_{i-1}\}}(v_i)\leq 6$. Let $G''=G-\{v_1,\ldots,v_{r-p}\}$.
Again, $G''$ is semi-planar (and planar if $G$ is planar) and $|G''|=n-r+p$ is divisible by $r$, so $G''$ has an equitable $r$-coloring $f'$. For $j=r-p,r-p-1,\ldots,1,$ we color
$v_j$ with color $\alpha_j$ distinct from the colors of its colored neighbors and from $\alpha_{j+1},\alpha_{j+2},\ldots,\alpha_{r-p}$. Since $p\geq 5$, for $j\geq 2$, $v_j$ has at most $6$ colored neighbors, and the number of already used $\alpha_i$ is 
$r-p-j\leq r-p-2$, we can find such $\alpha_j$ for each $j\geq 2$. For $j=1$, we have $d(v_1)\leq 5$ and the number 
of already used $\alpha_i$ is 
$r-p-1$. Thus, we get an equitable $r$-coloring of $G$.
\qed

We now describe the common setup for  proofs of both Theorems~\ref{clww} and~\ref{pla8}.
By Lemma~\ref{equ}, it is enough to consider graphs with $n=r s$ vertices for some $s\geq 1$.
We  use induction on $\|G\|$. 
If $G$ has no edges, the claim is trivial. 
So, let $G$ be an edge-minimal $n$-vertex semi-planar (or planar) graph $G$ with $\Delta(G)\leq r$  that  is not equitably $r$-colorable. It may have isolated vertices. Let $V_0$ denote the set of such vertices and $n_0=|V_0|$.
Let $x$ be a vertex of a minimum  degree in $G-V_0$ (we say that $d(x)=\delta^*(G)$) and let $y$ be any neighbor of $x$.
 By Lemma~\ref{pla}(a), 
 either $d(x)\leq 5$ or $\Delta(G)=6$. 
 As in the proof of Lemma~\ref{equ}, if $\Delta(G)=6$, then
we are done by the Hajnal-Szemer\' edi Theorem, so we may assume $d(x)\leq 5$.
 
  By induction hypothesis, $G-xy$ has an equitable $r$-coloring, say $\varphi$. 
  If vertices $x$ and $y$ are in different color classes, then $\varphi$  is also an equitable $r$-coloring of $G$. 
  Thus, we may assume that  the color classes of $G-x$ are $V_1,\dots,V_r$, where $|V_2|=\ldots=|V_r|=s$, $|V_1|=s-1$, and $y\in V_1$.  
  We call such (partial) colorings of $G$ {\em almost equitable}.

Define an auxiliary digraph $\mathcal{H}$ with the vertex set $\{V_1\dots,V_r\}$ where a directed edge $V_{i}V_{j}$ exists if and only if some vertex $v\in V_i$ has no neighbor in $V_j$. 
In order not to mix up vertices and edges in $\HH$ and $G$,
we will call the vertices in $\HH$  {\em classes} and edges 
in $\HH$  {\em  arcs}. 
We say that $v$ {\em witnesses} the  arc $V_{i}V_{j}$, and vertex $v$ is \emph{movable to $V_j$}. A class $V_i$ is {\em reachable} from class $V_j$ if $\HH$ contains a path from $V_j$ to $V_i$.
Naturally, a class $V_i$ is {\em reachable} from a set $\FF$ of classes, if it is reachable from at least one of classes in $\FF$.
Call a class $V_j$ {\em accessible} if $V_1$ is reachable from $V_j$, i.e., $\mathcal{H}$ contains a path from $V_j$ to $V_1$.
 Let $\mathcal{A}$ be the set of accessible classes in $\mathcal{H}$, and $\mathcal{B}$ be the set of classes not in $\mathcal{A}$.
 Among all almost equitable colorings, choose a coloring $\varphi$ with maximum $|\mathcal{A}|$.
 
 Set $a=|\mathcal{A}|$, $b=|\mathcal{B}|$, $A=\bigcup\mathcal{A}$ and $B=\bigcup\mathcal{B}$. Then $a+b=r$. 
 Also for each  $U\in\BB$ and each $V\in \AA$, every $u\in U$ has a neighbor in $V$, and hence
\begin{equation}\label{ff-1}
\mbox{\em for each  $U\in\BB$ and each $V\in \AA$,}\qquad |E_{G-x}(U,V)|\geq |U|=s.
\end{equation}

 By Lemma~\ref{pla}(b) applied to the bipartite graph formed by the edges of $G-x$ connecting $A$ with $B$, this yields
 \begin{align}\label{e(A,B)}
    a\cdot b \cdot s\leq|E_{G-x}(B,A)|\leq 2(|A|+|B|)=2(rs-1).
\end{align}

For distinct classes $X,Y\in\mathcal{A}$, we say $X$ \emph{blocks} $Y$ if $V_1$ is not reachable from $Y$ in $\HH-X$. 
A class in $\AA$ is  \emph{terminal} if it blocks no any other class in $\mathcal{A}$. In particular, if $\AA=\{V_1\}$, then
$V_1$ is terminal.
Let $\AA'$ be the set of terminal classes 
in $\AA$, $A'=\bigcup\mathcal{A}'$ and $a'=|A'|$.

Let $\mathcal{D}(x)$ be the set of classes with no neighbors of $x$. Since $d(x)\leq 5,$ $|\mathcal{D}(x)|\geq r-5$.
If $V_i\in\mathcal{A}\cap \mathcal{D}(x)$, then $\mathcal{H}$ contains a $V_i,V_1$-path, say $V_{i_1},V_{i_2},\ldots,V_{i_t}$, where $i_1=i$ and $i_t=1$. Moving $x$ into $V_i$, and each witness $v_{i_j}$ of $V_{i_j}V_{i_{j+1}}$ to $V_{i_{j+1}}$ along the path yields an equitable $r$-coloring of $G$. So,  $\mathcal{D}(x)\subseteq \mathcal{B}$; in particular 
\begin{equation}\label{br5}
    b=|\mathcal{B}|\geq r-5.
\end{equation}

For an edge $vu\in E_G(A,B)$ with $v\in V\in \AA$ and $u\in B$,  if $N_V(u)=\{v\}$, then we 
 say that $u$ and $v$ are \emph{solo neighbors} of each other, and each of them is a \emph{solo vertex}.

For $v\in A,$ let  $\FF_0(v)$ be the set of classes in $\BB$ that do not have neighbors of $v$.
Call a vertex $u\in V_i\in\mathcal{A}'$ {\em ordinary} if
some $u'\in V_i-u$ is movable to another class in $\AA$ or $a\leq 2$.

For $v\in A,$ let $Q(v)$ denote the set of solo neighbors of $v$ in $B$ and let $q(v)=|Q(v)|$.
Let $Q'(v)$ denote the set of vertices $u\in Q(v)$ that have non-neighbors in $Q(v)-u$ and let $q'(v)=|Q'(v)|$. 
We will use the following fact.

\begin{lem}\label{nB2}
Let $v\in V_i\in\mathcal{A}'$ be an ordinary vertex.
Let $u\in Q'(v)$, say $u\in W_j\in \BB$.

(a)   $|N(v)\cap W_j|\neq 1$.

(b) If $\FF_0(v)\neq \emptyset$, then $W_j$ is not reachable from $\FF_0(v)$.
\end{lem}

{\bf Proof.} Since $u\in Q'(v)$, there is some $u'\in Q'(v)$ not adjacent to $u$, say $u'\in W_{j'}\in \BB$.

Suppose first that (a) does not hold, i.e., $N(v)\cap W_j= \{u\}$. If some $v'\in V_i-v$ is movable to another class in $\AA$ or $a=1$, then we let
 coloring $\varphi'$ be obtained from $\varphi$ by moving $v$ to $W_j$ and $u$ to $V_i$. Each class in $\AA-V_i$ remains accessible as $V_i$ is a terminal class. And by the case,
the class $V_i-v+u$ is still accessible. Moreover,
now the class $W_j'$ containing $u'$ is also accessible with $u'$ becoming a witness, which contradicts the maximality of $a$. 

If $a=2$ and no $v'\in V_i-v$ is movable to another class in $\AA$, then since $V_i\in \AA'$, $i=2$ and $v$ is the unique vertex in $V_2$ movable to $V_1$. Then we consider
$\varphi''$  obtained from $\varphi$ by moving $v$ to $V_1$.
In this coloring, $V_2-v$ is the small class, and $v$ is a witness that $V_1+v$ is accessible. Moreover, both $W_j$ and $W_{j'}$ are now also accessible. This  contradiction proves (a).

The proof of (b) is similar. Moreover, the case when $a=2$ and no $v'\in V_i-v$ is movable to another class in $\AA$ word by word repeats the previous paragraph. So suppose (b) does not hold and either some $v'\in V_i-v$ is movable to another class in $\AA$ or $a=1$. This means there is
 $W_1\in \FF_0$ and
   $\HH$ contains a directed $W_{1},W_j$-path $P$.
 If $W_{j'}$ is a vertex in    $P$ distinct from $W_j$, then we switch the roles of $u$ and $u'$; thus we assume this is not the case.
  By renaming the classes in $\BB$, we may assume
   $P=W_{1},W_{2},\ldots,W_{\ell}$.
 For $h=1,2,\ldots,\ell-1$, let $u_h$ be a witness for the  arc $W_{h}W_{{h+1}}$.

 Change $\varphi$ as follows. Move $v$ to $W_1$, then for $h=1,2,\ldots,\ell-1$, move $u_h$ from $W_{h}$ to $W_{{h+1}}$, 
 and finally move $u$ to $V_i$.
 Call  the resulting coloring $\psi$. See Figure~\ref{fig:shifting-witnesses}. Each class in $\AA-V_i$ remains accessible as $V_i$ is a terminal class. And by the case,
the class $V_i-v+u$ is still accessible. Moreover, 
if $j'\neq j$ then class $W_{j'}$ is also accessible, and if $j'= j$ then 
 class $W_j-u+u_{\ell-1}$ is  accessible with $u'$ being a witness in both cases.
This proves Lemma~\ref{nB2}.\qed

\begin{figure}[h]
    \centering
 \includegraphics{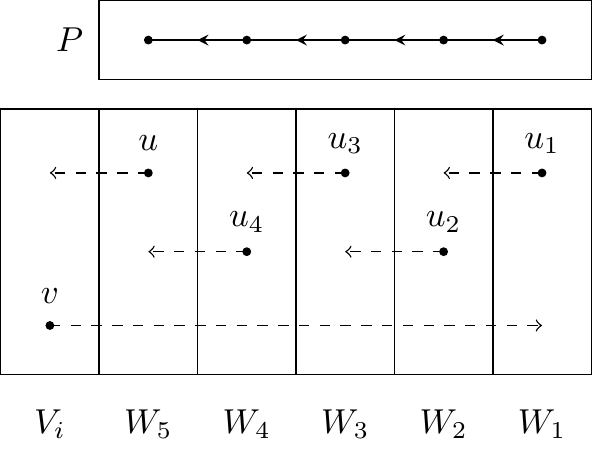}
    \caption{Obtaining $\psi$ in the proof of Lemma~\ref{nB2}(b), with $\ell=5$}
    \label{fig:shifting-witnesses}
    \end{figure}

\medskip
 For an arbitrary class $V\in\mathcal{A}$ and a vertex $u\in B$, let $\|V,u\|$ denote the number of edges incident to $u$ and a vertex in $V$. For each $u\in B$ and $v\in V\in \mathcal{A}$,
define the weights 
\begin{equation}\label{defw}
 w(v,u)=\frac{1}{\|V,u\|}\qquad \mbox{\em and}  \quad w(v)=\sum\nolimits_{uv\in E(G): u\in B}w(v,u).  
\end{equation}

By definition,
\begin{align}\label{r-2}
 \sum\nolimits_{v\in V}w(v)=   \sum\nolimits_{v\in V,u\in B}w(v,u)= |B|=bs. 
\end{align}

\section{Proof of Theorem~\ref{clww}}\label{semi}   
 
 For semi-planar $G$, we provide a bound on $q'(v)$ in terms of $q(v)$.

\begin{cl}\label{cl1}
    If  $q(v)\geq 8$, then $q'(v)\geq 5$. Also, if $q(v)= 7$, then $q'(v)\geq 4$. 
\end{cl}

{\bf Proof of claim.} Let $q=q(v)$ and $q'=q'(v)$. Consider graph $F=G[Q(v)\cup \{v\}]$. Vertices in $Q'(v)$ are those having degree less than 
$q-1$ in $F$. Then $|E(F)|\geq {q+1\choose 2}-{q'\choose 2}$. So, if $q'\leq 4$ and $q\geq 8$, then
$$|E(F)|\geq \frac{q}{2}(q+1)-{4\choose 2}\geq 4(q+1)-6>3(q+1),
$$
contradicting Lemma~\ref{pla}(a). If $q'\leq 3$ and $q=7$, then similarly
$|E(F)|\geq \frac{7}{2}(q+1)-3= 3.5(q+1)-3>3(q+1),
$ a contradiction again. This proves the claim.





\bigskip
When $r\ge 9$, by~\eqref{br5},  $b\ge r-5\ge 4$; thus $a\le 5$. If $ 3\leq a\leq 5$, then 
$$abs-2(rs-1)= a(r-a)s-2rs+2=(a-2)rs-a^2s+2>0,$$ contradicting~\eqref{e(A,B)}. Thus, $a\le 2$.

\medskip
We now show a helpful property of the auxiliary digraph $\mathcal{H}$. 

\begin{cl}\label{44'}
    The digraph $\mathcal{H}[\BB]$   has a strong component of order at least $r-2$.
\end{cl}

\begin{proof} Since $|\BB|\geq r-2\geq 7$, if each strong component of $\mathcal{H}[\BB]$ has at most $r-3$ vertices, then
the union $\mathcal{U}$ of some strong components of $\mathcal{H}$ has at least $3$ and at most $r-3$ vertices.

Suppose $|\UU|=m$.   Then  for every pair  $(U_i,W_j)$ where $U_i\in \UU$ and $W_j\in \BB-\UU$,
    either $U_iW_j$ is not an  arc or $W_jU_i$ is not an  arc in $\mathcal{H}$. 
    By the construction of $\mathcal{H}$, either every vertex in $W_j$ has at least one neighbor in $U_i$, or every vertex in $U_i$ has at least one neighbor in $W_j$. In both cases, $|E_{G-x}(U_i,W_j)|\ge \min\{|W_j|,|U_i|\}=s$. Also by~\eqref{ff-1}, $|E_{G-x}(U_i,A_j)|\ge s$ for each $A_j\in \AA$.

It follows that
denoting $U=\bigcup_{U_i\in \UU} U_i$ and $W=V(G)-U-x$, we have 
$$|E_{G-x}(U,W)| \ge m s (r-m) \geq 3 s(r-3)=2rs+(r-9)s.$$ 
For $r\geq 9$ this is greater than $2(rs-1),$ which contradicts Lemma~\ref{pla}(b) applied to the bipartite graph formed by the edges of $G-x$ connecting $U$ with $W$.
\end{proof}

Now we can prove the theorem. Recall that $1\leq a\leq 2$.


\medskip
{\bf Case 1:} $a=2$.
Let $\mathcal{A}=\{V_1,V_2\}$ and $\mathcal{B}=\{W_1,...,W_{r-2}\}$.
First, we show that
\begin{equation}\label{V1}
   \mbox{\em if some $v\in V_2$ has a solo neighbor in $B$, then $v$ also has a neighbor in $V_1$.}
\end{equation}
Indeed if $v\in V_2$ has a solo neighbor  $u\in W_j\in \mathcal{B}$, 
 then we consider a new coloring $\varphi'$ obtained from $\varphi$ by moving $v$ to $V_1$. The new almost equitable coloring has the small class $V_2-v$, and 
this class is reachable in the corresponding digraph $\mathcal{H}'$ from $V_1+v$ (with a witness $v$) and from $W_j$ (with a witness $u$). This contradiction to the maximality of $\mathcal{A}$ in $\varphi$ proves~\eqref{V1}. 

Since $V_2\in \mathcal{A}$, it contains a vertex $u$ with no neighbors in $V_1$. By~\eqref{V1},
$u$ has no solo neighbors in $B$, and hence $w(u)\leq d(u)/2<r-2$. Since by~\eqref{r-2}, the average weight of vertices in $V_2$ is $r-2$, this implies, that for some $v_0\in V_2$ we have $w(v_0)>r-2$.
By definition, 
$$w(v_0)\leq q(v_0)+\frac{1}{2}(|N(v_0)\cap B|-q(v_0))=\frac{1}{2}|N(v_0)\cap B|+\frac{1}{2}q(v_0)$$
Again by~\eqref{V1}, $v_0$ has a neighbor  in $V_1$ and so 
$|N(v_0)\cap B|\leq r-1$. Hence, in order to have $w(v_0)>r-2$, we need $|N(v_0)\cap B|= r-1$ and $q(v_0)\geq r-2$.

Let $\FF_0=\FF_0(v)$ is the set of classes in $\BB$ that do not have neighbors of $v_0$.  By Lemma~\ref{nB2}(a) and Claim~\ref{cl1}, 
among the $r-1$ neighbors of $v_0$ in $B$, at least $4$ vertices are not unique neighbors of $v_0$ in their color classes.
 It follows that 
\begin{equation}\label{ff0}
|\FF_0|\geq (r-2)-(r-1-\frac{4}{2})= 1.
\end{equation}

 By Claim~\ref{44'}, every color class in $\BB$ is reachable from $\FF_0(v)$.
 But there is some $u\in Q'(v_0)\cap W$ where $W\in\BB$, and with $W$ reachable from $\FF_0(v)$, we have a contradiction to Lemma~\ref{nB2}(b).

     \medskip
{\bf Case 2:} $a=1$.
 This case is similar to Case 1, but more complicated.
We may assume $\mathcal{A}=\{V_1\}$ and $\mathcal{B}=\{W_1,...,W_{r-1}\}$. 
Since $|V_1|=s-1$, by~\eqref{r-2}, the average weight of a vertex in $V_1$ is $\frac{(r-1)s}{s-1}>r-1$. Fix a vertex $v_0\in V_1$
with $w(v_0)>r-1$. For this we need $d(v_0)=r$ and $q(v_0)\geq r-1$. By  Claim~\ref{cl1},  $q'(v_0)\geq 5$.

Recall $\FF_0$ 
 as in Case 1.
 By Lemma~\ref{nB2}(a) and Claim~\ref{cl1}, 
among the $r$ neighbors of $v_0$ in $B$, at least $5$ vertices are not unique neighbors of $v_0$ in their color classes.
So, similarly to~\eqref{ff0}, we get
  \begin{equation}\label{ff2}
|\FF_0|\geq (r-1)-(r-\bigg\lceil\frac{5}{2}\bigg\rceil)= 2.
\end{equation}
 By Claim~\ref{44'} and~\eqref{ff2},  we have the following cases.

\medskip
{\bf Case 2.1:} Every color class in $\BB$ is reachable from $\FF_0$.
There is some $u\in Q'(v_0)\cap W$ where $W\in\BB$, and with $W$ reachable from $\FF_0$, we have a contradiction to Lemma~\ref{nB2}(b).

\medskip
{\bf Case 2.2:} Exactly one color class in $\BB$, say $W_{r-1}$ is not reachable from $\FF_0$.
If there is $u\in Q'(v_0)\cap W$ where $W\in\BB\setminus\{W_{r-1}\}$, then we again have a
contradiction to Lemma~\ref{nB2}(b).
So, assume $Q'(v_0)\subseteq W_{r-1}$. Consider the following new weight function $w'$.

For each $u\in B-W_{r-1}$ and $v\in V_1$,define the weight $w'(v,u)=w(v,u)=\frac{1}{\|V_1,u\|}$, but for $u\in W_{r-1}$ and $v\in V_1$ we let $w'(v,u)=\frac{1}{2}w(v,u)=\frac{1}{2\|V_1,u\|}$. Then for each $v\in V_1$, define
$$w'(v)=\sum\nolimits_{uv\in E(G): u\in B}w'(v,u).$$



By definition,
\begin{align}\label{r-1}
 \sum\nolimits_{v\in V_1}w'(v)=   \sum\nolimits_{v\in V_1,u\in B}w'(v,u)= (r-1.5)s. 
\end{align} 
Since $|V_1|=s-1$,  the average new weight of a vertex in $V_1$ is $\frac{(r-1)s}{s-1}>r-1.5$. Fix a vertex $v'\in V_1$
with $w'(v')>r-1.5$. Since $Q'(v_0)\subseteq W_{r-1}$ and $q'(v_0)\geq 5$, we have $w'(v_0)\leq (q(v_0)-q'(v_0))+\frac{1}{2}(r-q(v_0)+q'(v_0))\leq r-2.5.$ Thus $v'\neq v_0$.
Since $a=1$, $v'$ is  ordinary.

By Lemma~\ref{44'}, we may assume the following:
 \begin{equation}\label{nw1}
\mbox{\em For all $W_i,W_j$ such that $1\leq i,j\leq r-2$, $\HH$ has a $W_i,W_j$-path.
}
\end{equation}

Let $\hat{Q}(v')=Q(v')-W_{r-1}$. If $|\hat{Q}(v')|=m\leq r-3$, then $w'(v')\leq m+\frac{1}{2}(r-m)\leq r-3+\frac{3}{2},$ a contradiction. Thus $|\hat{Q}(v')|\geq r-2\geq 7$. So, by  Claim~\ref{cl1},  $|Q'(v')|\geq 4$. Since 
by Lemma~\ref{nB2}(a), 
among the $r$ neighbors of $v'$ in $B$, at least $4$ vertices are not unique neighbors of $v'$ in their color classes, similar to~\eqref{ff0},
$|\FF_0(v')|\ge 1$.
Choose a smallest $m$ such that $W_m\in \FF_0(v')$.

\medskip
{\em Case 2.2.1:}  $1\leq m\leq r-2$. 
If some $4$ vertices in $Q'(v')$ are in $W_{r-1}$, then
$w'(v')\leq r-4(1/2)=r-2$, a contradiction. Thus some $u\in Q'(v')$ is not in $W_{r-1}$. 
Say $u\in W_j\in \BB-\FF_0(v')$.
Then by Lemma~\ref{nB2}(b), $W_j$ is not reachable from $W_m$, but this is a contradiction to~\eqref{nw1}.

\medskip
{\em Case 2.2.2:} $m=r-1$. By the minimality of $m$ and by Lemma~\ref{nB2}(a), in this case $q'(v')=4$, and these four vertices are in exactly two color classes. Since $q(v')\geq r-2\geq 7$, there is $z_0\in Q(v')-Q'(v')$. This $z_0$ is adjacent to $v'$ and to at least $r-3$ vertices in $Q(v')$, and hence has at most $2$ neighbors in $W_{r-1}$. Recall that at least $5$ vertices in $Q'(v_0)$, say
$z_1,\ldots,z_5$, are in $W_{r-1}$. So, we may assume that $z_0$ is not adjacent to $z_1,z_2$ and $z_3$.

By~\eqref{ff2}, we may assume that $v_0$ has no neighbors in $W_1$. Let $W(z_0)$ be the class containing $z_0$.
By~\eqref{nw1}, $\HH$ has a $W_1,W(z_0)$-path, say $W_1,W_2,\ldots,W_\ell$, where $W_\ell=W(z_0)$.
For $j=1,2,\ldots,\ell-1$, let $u_j$ be a witness for the  arc  $W_{j}W_{{j+1}}$.

Consider a new coloring $\varphi'$ obtained as follows. Move $v'$ to $W_{r-1}$, then $z_1$ to $V_1-v'$, then $v_0$ to $W_1$,  then for $j=1,2,\ldots,\ell-1$, move $u_j$ from $W_{j}$ to $W_{{j+1}}$, 
 and finally move $z_0$ to $V_1$. Since $z_0z_1\notin E(G)$ and $v'$ has no neighbors in $W_{r-1}$,
 $\varphi'$ is an almost equitable coloring of $G-x$. But now the class $W_{r-1}-z_1+v'$ is accessible with a witness $z_2$, contradicting the maximality of $a$.\qed

\section{Proof of Theorem~\ref{pla8}}

By Theorem~\ref{clww}, it is enough to consider the case $r=8$.
Since $G$ is planar, we can give a better bound on $q'(v)$ in terms of $q(v)$.

\begin{cl}\label{q'}
    If $q(v)\ge 5$, then $q'(v)\ge q(v)-1$.
\end{cl}

\begin{proof}
    Assume that $q'(v)\le q(v)-2$. Then there are two solo neighbors $u_1,u_2$ of $v$  adjacent to all other vertices in $Q(v)$.
    In particular, $G$ contains $K_{3,q(v)-2}$ with parts $\{v,u_1,u_2\}$ and $Q(v)-\{u_1,u_2\}$, a contradiction to planarity of $G$.
\end{proof}  

\medskip
We now prove an analogue of Claim~\ref{44'} on strong components of $\HH$.

\begin{cl}\label{44}
   Suppose $a=|\mathcal{A}|\le 4$.

   (i) No union of some strong components of $\HH$ has exactly $4$ vertices.

   (ii) Digraph $\mathcal{H}$  either has a strong component of size at least $5$, or has two strong components of size $3$ and one strong component of size $2$. 
\end{cl}

\begin{proof} Suppose (i) does not hold, and
the union of some strong components of $\mathcal{H}$ consists of exactly  $4$ classes, say this union is $\mathcal{U}=\{U_1,U_2,U_3,U_4\}$. Let  $\mathcal{W}=V(\HH)-\mathcal{U}=\{W_1,W_2,W_3,W_4\}.$
    Then as in the proof of  Claim~\ref{44'}, $|E_G(U_j,W_i)|\ge \min\{|W_i|,|U_j|\}$. Without loss of generality, assume that $|U_1|=|V_1|=s-1$.
Denoting $U=\bigcup_{i=1}^4 U_i$ and $W=\bigcup_{j=1}^4W_j$, we have 
$$|E_G(U,W)|=|E_G(U_1,W)|+\sum_{i=2}^4|E_G(U_i,W)|\ge 4(s-1)+12s=16s-4>2(8s-1)-4,$$ contradicting Lemma~\ref{pla}(b). Thus (i) holds.

 Let the sizes of the strong components of $\HH$ be $a_1,\ldots,a_m$ and $a_1\geq a_2\geq\ldots\geq a_m$. 
Then $a_1+\ldots+a_m = 8$.
If (ii) does not hold, then $a_1\leq 4$. Moreover, by (i), $a_1\leq 3$ and no sums of several $a_i$ equal to $4$.
This is possible only if $a_1=a_2=3$ and $a_3=2$.
\end{proof}

Notice that by the way we define $\AA$ and $\BB$, each strong component in $\HH$ should be contained in either $\AA$ or $\BB$.

\medskip
With $r=8$, by~\eqref{br5} we have  $b\ge 3$.
So $a=r-b\le 5$.
By Claim~\ref{44},  $a=4$ would lead to a contradiction. Thus it suffices to consider the cases when $a=1,2,3$ and $5$.






\subsection{Proof of the case $a=1$}


Recall the weight functions $w(v,u)$ and $w(v)$ defined by~\eqref{defw}.
By~\eqref{r-2}, with $b=r-a=7$ and $|A|=|V_1|=s-1$, there is some $v_0\in V_1$ with $d(v_0)\ge w(v_0)\ge 7s/(s-1)>7$. Thus $d(v_0)=8$.
Note that $N(v_0)\subseteq B$. If $q(v_0)\le 6$, then $w(v_0)\le q(v_0)+(d(v_0)-q(v_0))/2=4+q(v_0)/2\le 7$, a contradiction, so $q(v_0)\ge 7$ and $q'(v_0)\ge 6$.

Let $\FF_0$ denote the set of classes in $\BB$ that do not have neighbors of $v_0$, $\mathcal{F}$ denote the set of classes reachable in $\HH$ from $\FF_0$, $f=|\FF|$ and  $F=\bigcup \FF$. 
Notice that every color class $V_i$ is trivially reachable from itself in $\HH$, so $\FF_0\subseteq\FF$.
By Lemma~\ref{nB2}(a) with $q'(v_0)\ge 6$, at least $6$ vertices in $N(v_0)$ are not unique neighbors of $v_0$ in their color classes. It follows that 
\begin{equation}\label{f7}
    7\ge f\ge |\mathcal{F}_0|\ge (r-1)-(r-\frac{6}{2})= 2.
\end{equation}

{\bf Case 1.1:} $f=2$, say $\FF=\{V_2,V_3\}$. In this case, by~\eqref{f7}, $\mathcal{F}=\mathcal{F}_0$. Then  $q'(v_0)=6$ and there are three classes $V_6,V_7,V_8$ such that $Q'(v_0)=N(v_0)\cap (V_6\cup V_7\cup V_8)$. Specifically, by Lemma~\ref{nB2}(a), we get $|N(v_0)\cap V_i|=|Q'(v_0)\cap V_i|=2$ for $i\in\{6,7,8\}$. Since $q(v_0)\ge 7>q'(v_0)$,  some vertex $v'\in Q(v_0)$ is adjacent to all of $Q'(v_0)$.

Let $N(v_0)\cap V_8=\{w,w'\}$. Consider the coloring $\varphi''$ of $G-x$ obtained from $\varphi$ by moving $v_0$ into $V_8$ 
and moving $w$ and $w'$ into $V_1-v_0$. Denote $V_1'=(V_1-v_0)\cup \{w,w'\}$ and $V_8'=(V_8- \{w,w'\})\cup\{v_0\}$.
If $x$ is not adjacent to $V'_8$, then we extend $\varphi''$ to $G$ by moving $x$ into $V'_8$. 
This extension is an equitable coloring of $G$ as $|V'_1|=|V'_8\cup\{x\}|=s$ while other color classes remain unchanged.
Thus we may assume that $x$ has a neighbor $y'$ in $V'_8$. 

Note that $\varphi''$ is an almost equitable coloring of $G-x$ with the small class $V'_8$.
 By the maximality of $a$, every vertex in $V(G)-V'_8$ has a neighbor in $V'_8$.
Thus 
\begin{equation}\label{v6}
    |E_G(V_8', V_i)|\ge |V_i|=s\qquad \mbox{for all $i\in [7]-\{1\}$}.
\end{equation} 

Now we count the edges between $X=V_1\cup V_8\cup F$ and $Y=V(G)-x-X=V_4\cup V_5\cup V_6\cup V_7 $. 
Since $a=1$ and $f=2$, for color classes $F_i\in \FF$ and $B_j\in \BB\setminus\FF$, there is no edge of the form $F_iB_j$ or $B_jV_1$ in $\HH$. Thus 
\begin{equation}\label{12s}
    |E_G(V_1\cup F, Y)|\ge |E_G(V_1,Y)|+|E_G(F,Y)|\ge 4s+8s=12s.
\end{equation}

Further notice that 
\begin{equation}\label{ge6}
    |E_G(V_8', Y)\cap E_G(V_1,Y)|=|E_G(v_0,N(v)-\{w,w'\})|=6,
\end{equation}
and that $$|E_G(v',\{w,w'\})\cap(E_G(V_8', Y)\cup E_G(V_1,Y))|=0.$$
Thus by~\eqref{v6}, we get $$|E_G(V_8,Y)|\ge|E_G(V_8', Y)|-|E_G(v_0,N(v)-\{w,w'\})|+|E_G(v',\{w,w'\})|\ge 4s-6+2=4s-4.$$
Combining this with~\eqref{12s}, we obtain $$|E_G(X,Y)|\ge 12s+4s-4=16s-4>2(8s-1)-4,$$ a contradiction to Lemma~\ref{pla}(b).

\medskip
{\bf Case 1.2:} $f\in \{3,4\}$. In this case we do not have a  strong  component of size at least $5$ in $\HH$, and $V_1$ forms a  strong  component of size $1$ by itself. Then we have a contradiction to Claim~\ref{44}.

\medskip
{\bf Case 1.3:} $f=5$. Let $\BB-\FF=\{V_2,V_3\}$ and $C=V_2\cup V_3$. Similarly to Case 2.2 in Section~\ref{semi}, we have $Q'(v_0)\subseteq C$. Consider the following new weight function $w'$.

For each $u\in B\setminus C=F$ and $v\in V_1$, define  $w'(v,u)=w(v,u)=\frac{1}{||V_1,u||}$, but for $u\in C$ and $v\in V_1$, let $w'(v,u)=\frac{1}{2}w(v,u)=\frac{1}{2||V_1,u||}$. For each $v\in V_1$, define 
$$w'(v)=\sum\nolimits_{uv\in E(G):u\in B}w'(v,u).$$
By definition, $\sum_{v\in V_1}w'(v)=\sum_{v\in V_1,u\in B}w'(v,u)=6s$.

Since $|V_1|=s-1$, the average new weight of a vertex in $V_1$ is $6s/(s-1)>6$. Pick a vertex $v'\in V_1$ with $w'(v')>6$. 
Let $Q_1(v')=Q(v')\cap F$ and $q_1(v')=|Q_1(v')|$. 
By definition, for $u\in Q(v')-Q_1(v')$, $w'(v',u)\le\frac{1}{2}$. 
Thus  $$6<w'(v')\le q_1(v')+\frac{1}{2}(8-q_1(v'))=4+\frac{q_1(v')}{2},$$ so $q_1(v')\ge 5$.
Denote by $Q_1'(v')$ the set of vertices $u\in Q_1(v')$ that have non-neighbors in $Q_1(v')-u$ and $q'_1(v')=|Q_1'(v')|$.

\medskip
{\em Case 1.3.1:} $|N(v')\cap C|\ge 1$.  
Suppose first that every class in $\FF$ has a neighbor of $v'$. Let $\FF'(v')$ denote the set of classes in $\FF$ that contain  vertices in $Q'_1(v')$ and no other neighbors of $v'$.
Since $q(v')\geq q_1(v')\geq 5$, repeating the argument of  Claim~\ref{q'}, we get
$q'_1(v')\geq q_1(v')-1$. So since $|N(v')\cap F|\le 7$, $|\FF'(v')|\geq 2$. By Lemma~\ref{nB2}(a), each class in $\FF'(v')$ has at least two vertices
from $Q'_1(v')$.
If each of them has at least $3$ such vertices, then $N(v')\cap F$ has at least $3|\FF'(v')|+(5-|\FF'(v')|)=2|\FF'(v')|+5$ vertices.
But this contradicts the fact that $|N(v')\cap F|\le 7$ and $|\FF'(v')|\geq 2$.
Thus,
some color class $V_8\subseteq F$ satisfies
$|V_8\cap Q_1'(v')|=|V_8\cap N(v')|=2$, say $V_8\cap Q_1'(v')=\{z,z'\}$. 

Similarly to Case 1.1, we consider a coloring $\varphi''$ of $G-x$ obtained from $\varphi$ by moving $v'$ into $V_8$ 
and moving $z$ and $z'$ into $V_1-v$. Denote $V_1'=(V_1-v')\cup \{z,z'\}$ and $V_8'=(V_8- \{z,z'\})\cup\{v'\}$.
As in Case 1.1, $\varphi''$ is an  equitable coloring of $G-x$ with the small class $V'_8$.
 By the maximality of $a$, every vertex in $V(G)-V'_8$ has a neighbor in $V'_8$.
Thus~\eqref{v6} holds again.

Now we count the edges between $X=V_1\cup V_8\cup C$ and $Y=V(G)-x-X=V_4\cup V_5\cup V_6\cup V_7 $. 
Similarly to~\eqref{12s}, we get
$$|E_G(V_1\cup C, Y)|\ge 12s.$$

As $v'$ has a neighbor in $C$, $|N(v')\cap F|\le 7$.
So similarly to~\eqref{ge6}, we have 
$$|E_G(V_8', Y)\cap E_G(V_1,Y)|=|E_G(v',(N(v')\cap F)-\{z,z'\})|\le 5.$$ Hence 
$$|E_G(V_8,Y)|\ge |E_G(V_8', Y)-E_G(V_1,Y)|\ge 4s-5.$$
Therefore,
$$|E_G(V_1\cup C\cup V_8,Y)|\ge 16s-5>2(8s-1)-4,$$  a contradiction to Lemma~\ref{pla}(b).

Thus, we may assume that
 some class $U\in \FF$ contains no neighbors of $v'$.
Since $a=1$, by Claim~\ref{44}, $\HH$ has a  strong component $\HH_1$ of size at least $5$. Since $\HH$ has no edges from $\FF$ to $V_1,V_2$ or $V_3$,
the vertex set of $\HH_1$ is $\FF$.
Hence every class in $\FF$ is reachable from $U.$
In particular, there is some vertex $u\in Q_1'(v')$ that is contained in some class $V_j$ and $V_j$ is reachable from $U$.
However,  as $a\le 2$,  $v'$ is ordinary and this contradicts Lemma~\ref{nB2}(b).

\medskip
{\em Case 1.3.2:} $|N(v')\cap C|=0$. Using the argument of Claim~\ref{q'}, we can show that as $q_1(v')\ge 5$, $$q_1'(v')=|Q_1'(v')|\ge q_1(v')-1.$$
So, there is at most one class in $\FF$ containing the vertex from $Q_1(v')\setminus Q_1'(v')$ (if exists), at most $3$ classes containing vertices from $N(v')\setminus Q_1(v')$, and hence there is a
 class $V_8\in\FF$ with $V_8\cap (N(v')-Q_1'(v'))=\emptyset$.

If $V_8$ has no neighbors of $v'$ at all, then we can denote the class as $U$ and apply the argument at the end of Case 1.3.1  again. 
Otherwise, by Lemma~\ref{nB2}(a), $V_8$ has at least two vertices
from $Q'_1(v')$, say
  $\{w_1,w_2\}\subseteq V_8\cap Q_1'(v')$.

Recall that $q'(v)\ge 6$. Without loss of generality, assume that $\{v_1,v_2,v_3\}\subseteq V_2\cap Q'(v)$.
Since $G$ is planar, it is $K_{3,3}$-free, so by symmetry we can assume that $w_1$ and $v_1$ are not adjacent in $G$. 
Take $W_1\in\FF_0\subseteq \FF$. By Claim~\ref{44}, $\HH$ contains a $W_1,V_8$-path $P$. Let $P=W_{1},W_{2},\ldots,W_{\ell}$ where $W_{\ell}=V_8$.
 For $j=1,2,\ldots,\ell-1$, let $u_j$ be a witness for the  arc  $W_{j}W_{{j+1}}$.

 Change $\varphi$ as follows. Move $v_0$ to $W_1$, then for $j=1,2,\ldots,\ell-1$, move $u_j$ from $W_{j}$ to $W_{{j+1}}$, move $u$ to $V_2$, move $v_1$ to $V_1$ and finally move $w_1$ to $V_1$. Class $V_1-\{v_0,u\}+\{v_1,w_1\}$ remains accessible, but now $V_2-v_1+u$ is also accessible witnessed by $v_2$, contradicting the maximality of $a$.

\medskip
{\bf Case 1.4:} $f=6$. Similarly to the argument of Case 2.1 in Section~\ref{semi}, suppose $B-F=V_2$. Then $Q'(v)\subseteq V_2$. We pick two arbitrary sets $X_1,X_2\in \FF$. Let $\mathcal{X}$ be the collection of classes in $\HH$ reachable from $X_1$ and $X_2$.
Then  $2\le |\mathcal{X}|\le 6$, since both $V_1$ and $V_2$ are not  reachable from $X_1$ and $X_2$. Consider these cases.

\medskip
{\em Case 1.4.1:} $2\le |\mathcal{X}|\le 4$. As in Case 1.2, we do not have a strong component of size at least $5$ in $\mathcal{H}$, and $V_1$ must form a  strong  component of size $1$ by itself. 
Then we have a contradiction to Claim~\ref{44}.

\medskip
{\em Case 1.4.2:} $|\mathcal{X}|=5$. Assume that $V_3\in\FF\setminus\mathcal{X}$. Since $V_1$ forms a strong component in $\HH$, by
 Claim~\ref{44},   $\HH[\mathcal{X}]$  is strongly connected.

As in Case 1.3, let $C=V_2\cup V_3$. Let $w'(v,u)=w(v,u)=\frac{1}{||V_1,u||}$ for each $u\in B\setminus C$ and $v\in V_1$, but for $u\in C$ and $v\in V_1$, let $w'(v,u)=\frac{1}{2}w(v,u)=\frac{1}{2||V_1,u||}$. For each $v\in V_1$, define 
$$w'(v)=\sum\nolimits_{uv\in E(G):u\in B}w'(v,u).$$
By definition $\sum_{v\in V_1}w'(v)=\sum_{v\in V_1,u\in B}w'(v,u)=6s$.

Since $|V_1|=s-1$, the average new weight of a vertex in $V_1$ is $6s/(s-1)>6$. Pick vertex $v'\in V_1$ with $w'(v')>6$. We claim that we can repeat the argument from Case 1.3 with $v'$ and $C$ defined identically. Thus, both $|N(v')\cap C|\ge 1$ and $|N(v')\cap C|=0$ would lead to a contradiction.

\medskip
{\em Case 1.4.3:} $|\mathcal{X}|=6$.
Since $X_1,X_2$ were picked arbitrarily,
\begin{equation}\label{sco}
    \mbox{\em $\HH[\mathcal{F}]$  is strongly connected.}
\end{equation}


We again use the  function $w(v,u)=\frac{1}{||V,u||}$ defined by~\eqref{defw}.
For each $v\in V_1$, define 
$$w_6(v)=\sum\nolimits_{uv\in E(G):u\in F\setminus V_3}w(v,u).$$
By definition $\sum_{v\in V_1}w_6(v)=\sum_{v\in V_1,u\in F\setminus V_3}w(v,u)=6s$. 
Since $a\le 2$, $u$ is ordinary.

Since $|V_1|=s-1$, the average weight of a vertex in $V_1$ is $6s/(s-1)>6$. Pick vertex $u\in V_1$ with $w_6(u)>6$. Notice that $w_6(v_0)<|N(v_0)\setminus Q'(v_0)|\le 2$, so $u$ and $v_0$ are distinct. Let $Q_6(u)=Q(u)\cap (F\setminus V_3)$ and $q_6(u)=|Q_6(u)|$. Denote $Q_6'(u)$ the set of vertices $w\in Q_6(u)$ that have non-neighbors in $Q_6'(u)-w$.

If $|N(u)\cap V_2|\ge 1$, then $q_6(u)\ge 6$. By Claim~\ref{q'}, there is at most $1$ class in $\FF$ containing vertex from $Q_6(u)\setminus Q_6'(u)$, and at most $1$ class containing vertices from $N(u)\setminus Q_6(u)$. Thus by Lemma~\ref{nB2}(a), $|\FF_0(u)|\ge 2$. Pick $z\in Q_6'(u)$ and the color class of $z$ is $W(z)$.
Then by~\eqref{sco}, $W(z)$ is reachable from $\FF_0(u)$, but this is a contradiction to Lemma~\ref{nB2}(b).

If $|N(u)\cap V_2|=0$, then $q_6(u)\ge 5$. Since $G$ is planar, it is $K_{3,3}$-free.
Then there is some $u'\in Q_6'(u)$ that is not adjacent to some $v_1\in V_2\cap Q'(v)$. Let the color class of $u'$ be $U'$. Take $W_1\in\FF_0\subseteq \FF$. By~\eqref{sco}, $\HH$ contains a $W_1,U'$-path $P$. Let $P=W_{1},W_{2},\ldots,W_{\ell}$ where $W_{\ell}=U'$.
 For $j=1,2,\ldots,\ell-1$, let $u_j$ be a witness for the  arc  $W_{j}W_{{j+1}}$.

 Change $\varphi$ as follows. Move $v_0$ to $W_1$, then for $j=1,2,\ldots,\ell-1$, move $u_j$ from $W_{j}$ to $W_{{j+1}}$, move $u$ to $V_2$, move $v_1$ to $V_1$ and finally $u'$ to $V_1$. Class $V_1-\{v_0,u\}+\{v_1,u'\}$ remains accessible, but now $V_2-v_1+u$ is also accessible, witnessed by some $v_2\in V_2\cap Q'(v)$ distinct from $v_1$.  This contradicts the maximality of $a$.
 
 \medskip
 {\bf Case 1.5:} $f=7$.
 There is some $u\in Q'(v_0)\cap W$ where $W\in\BB=\FF$.
 But with $W$ reachable from $\FF_0$, we have a contradiction to Lemma~\ref{nB2}(b).

\subsection{Proof of the case $a=2$}
Let $\AA=\{V_1,V_2\}$.
 For each $u\in B$ and $v\in V_2$,
let $w(v,u)=\frac{1}{\|V,u\|}$. Then for each $v\in V_2$, let
$$w_2(v)=\sum\nolimits_{uv\in E(G): u\in B}w(u,v). $$
By definition $\sum_{v\in V_2}w_2(v)=\sum_{v\in V_2,u\in B}w(v,u)=6s$.

There is a movable vertex $v'\in V_2$, and by~\eqref{V1}, $v'$ has no solo neighbors in $B$.  So $w_2(v')\le 8\cdot\frac{1}{2}=4$. Then there is a vertex $v_0\in V_2$ with $w_2(v_0)>6$. Notice that such $v_0$ should not be movable in $A$, so $|N(v_0)\cap B|\le 7$. To have $w_2(v_0)>6$, we need $q(v_0)=|Q(v_0)|\ge 6$. For each class $U\in \BB$, by Lemma~\ref{nB2}(a), if $|Q'(v_0)\cap U|\neq 0$, then $|Q'(v_0)\cap U|\ge 2$. Thus there are distinct color classes $U_1,U_2\in \FF_0(v_0)$. Let $\mathcal{U}$ be the collection of  classes in $\BB$ reachable from $\FF_0(v_0)$. Then as $a=2$, $2\le |\mathcal{U}|\le 6$. If $|\UU|=2$, then $|\AA\cup \UU|=4$, contradicting Claim~\ref{44}(i). For the same reason, $|\UU|\neq 4$. The remaining cases are as follows.

\medskip

{\bf Case 2.1:} $|\mathcal{U}|=3$, say $\mathcal{U}=\{U_1,U_2,U_3\}$. Let 
$U=\bigcup \mathcal{U}$, $\mathcal{W}=\BB\setminus\mathcal{U}=\{W_1,W_2,W_3\}$ and $W=\bigcup\mathcal{W}$.
For $i=1,2$, let $M_i$ denote the set of vertices in $V_i$ movable to $V_{r-i}$. 
If $m_2\ge m_1+2$, we move a vertex from $M_2$ to $V_1$, and relabel $V_1$ as $V_2'$ and $V_2$ as $V_1'$. 
Then there are $m_2-1$  vertices movable from $V_1'$ to $V_2'$ and $m_1+1$ movable from $V_2'$ to $V_1'$. So, we may assume
\begin{equation}\label{m21}
    m_2\le m_1+1.
\end{equation}

Since no vertex in $U_i$ is movable to $W_j$ for $1\le i,j\le 3$,
\begin{equation}\label{9s}
    |E_G(U,W)|\ge |\mathcal{U}||\mathcal{W}|s=9s.
\end{equation} 

Suppose that there are $k_2$ isolated vertices in $V_2$. Since $d(x)=\delta^*(G)\ge 2$, 
$$|E_G(M_2,B)|\ge 2(m_2-k_2).$$
By the symmetry between $\UU$ and $\mathcal{W}$, we can assume 
$$|E_G(M_2, U)|\ge m_2-k_2.$$

 If for every vertex $z\in U$, $|N(z)\cap (V_2\setminus M_2)|\ge 1$, then 
 \begin{equation}\label{mk2}
     |E_G(V_2,U)|=|E_G(V_2\setminus M_2,U)|+|E_G(M_2,U)|\ge 3s+m_2-k_2.
 \end{equation}

Otherwise there is a vertex $z\in U$ with $|N(z)\cap (V_2\setminus M_2)|=0$.
Since $z$ is not movable to $V_2$, it is adjacent to some vertices in $M_2$. 
If there is any $v_1\in M_1$ that is adjacent to $y$ and $v_2\in M_2$ that is not adjacent to $z$, then we can switch $v_1$ and $v_2$ to increase $|N(z)\cap M_2|$.
When $|N(z)\cap M_2|$ is maximized in this way, we either have $M_2\subseteq N(z)$ or $|N(z)\cap M_2|<m_2$ and $|N(z)\cap M_1|=0$.
In the latter case, we can switch $N(z)\cap M_2$ with equal number of vertices in $M_1$ since $m_2\le m_1+1$. 
The switched vertices remain movable to the other class. 
However, $z$ would become movable to $V_2$ , since $z$ has no neighbor in $V_2$ after the switch, a contradiction to the maximality of $a$.
So  $M_2\subseteq N(z)$, and hence $|N(z)\cap M_2|= m_2$. 
Let $Z$ be the collection of all such $z\in U$ that $|N(z)\cap (V_2\setminus M_2)|=0$ and $|N(z)\cap M_2|= m_2$. If $k=|Z|$, then  \begin{equation}\label{k1}
    |E_G(V_2,U)|\ge 3s-k+km_2,
\end{equation}where $k,m_2\ge 1$.

Now we count the edges between $V_2\cup W$ and $V_1\cup U$, with $k_2$ isolated vertices in $V_2$ removed, so there should be $8s-1-k_2$ vertices. 
 When $k=0$, we use the bounds~\eqref{9s}, ~\eqref{mk2}, $|E_G(V_2,V_1)|\ge s-m_2$ and $|E_G(W,V_1)|\ge 3s$ to derive that 
 $$   |E_G(V_2\cup W,V_1\cup U)|=|E_G(V_2\cup W,U)|+|E_G(V_2\cup W,V_1)|$$
 $$\ge  (9s+3s+m_2-k_2)+(3s+s-m_2)=16s-k_2>2(8s-1-k_2)-4.
$$ 

 When $k\ge 1$, we use the bound~\eqref{k1} instead of~\eqref{mk2}:
$$    |E_G(V_2\cup W,V_1\cup U)|=|E_G(V_2\cup W,U)|+|E_G(V_2\cup W,V_1)|$$
$$\ge  (9s+3s-k_2+km_2)+(3s+s-m_2)\ge 16s-k_2>2(8s-1-k_2)-4.$$
In both cases we get a contradiction to Lemma~\ref{pla}(b).

 \medskip
{\bf Case 2.2:}  $|\mathcal{U}|=5$. Denote $\{V_3\}=\BB\setminus\mathcal{U}$. We should have $Q'(v_0)\subseteq V_3$.
Consider a new weight function $w_2'$ where $w_2'(v,u)=w_2(v,u)=\frac{1}{||V_2,u||}$ for $v\in V_2$ and $u\in B\setminus V_3$, but for $u\in V_3$, $w_2'(v,u)=\frac{1}{2||V_2,u||}$. For each $v\in V_2$, define 
$$w_2'(v)=\sum\nolimits_{uv\in E(G):u\in B}w_2'(v,u).$$
By definition $\sum_{v\in V_2}w_2'(v)=\sum_{v\in V_2,u\in B}w_2'(v,u)=\frac{11}{2}s$. 
Note that $w_2'(v_0)\le 5\cdot\frac{1}{2}+2=\frac{9}{2}<\frac{11}{2}$. Hence there is some $u\in V_2-v_0$ with $w_2'(u)>\frac{11}{2}$.
Then $|Q(u)
\setminus V_3|\ge 5$, so $|N(u)\cap V_1|\ge 1$ and by Claim~\ref{q'}, $q'(u)\ge q(u)-1$.

{\em Case 2.2.1:} $|N(u)\cap U_3|=0$ for some $U_3\in \mathcal{U}$. We have $2$ classes $V_1,V_2$ not reachable from the other $6$ classes and  class $V_3$ not reachable from the remaining $5$ classes. 
So, $\{V_3\}$ forms a strong component in $\HH$. Hence by
 Claim~\ref{44},  $\HH(\mathcal{U})$ is strongly connected. 
Take some $z\in Q'(u)\setminus V_3$ with color class $W(z)$.
Then in particular $W(z)$ is reachable from $\FF_0(v_0)$, but this is a contradiction to Lemma~\ref{nB2}(b). 

{\em Case 2.2.2:} $|N(u)\cap U|\ge 1$ for every $U\in \mathcal{U}$. As $|N(u)\setminus V_1|\le 7$, $|Q(u)\setminus V_3|\ge 5$ and
$|Q'(u)\setminus V_3|\ge |Q(u)\setminus V_3|-1\ge 4$, at most $3$ classes in $\UU$ contain vertices in $N(u)\setminus Q'(u)$.
So, by Lemma~\ref{nB2}(a), some two classes in $\UU$ contain at least two vertices in $Q'(u)$ each; thus at least $4$ together.
For this to happen, we need $|N(u)\setminus V_1|= 7$,
 $|Q(u)\cap V_3|=0$ and $Q(u)\setminus Q'(u')\neq \emptyset$, say
 $z\in Q(u)\setminus Q'(u')$. Note that $|N(z)\cap (B\setminus V_3)|\ge 4$, and $z$ is adjacent to $u$ and some vertex in $V_1$ by definition.
  Thus $|N(z)\cap Q'(v)|\le 2$, so there are $v_1,v_2\in Q'(v)$ that are not adjacent to $z$.
Let the color class of $z$ be $W(z)$. Pick $U_1\in\FF_0(v_0)$. Then there is a $U_1,W(z)$ path $P$. Let $P=W_1,W_2,\dots,W_\ell$ where $W_1=U_1,W_\ell=W(z)$. 
For $j=1,2,\dots,\ell-1$, let $u_j$ be a witness for the  arc  $W_jW_{j+1}$.

Now we change $\varphi$ as follows. Move $v_0$ to $U_1$, then for $j=1,2,\dots,\ell-1$, move $u_j$ from $W_j$ to $W_{j+1}$, move $z$ and $v_1$ to $V_2$, and finally $u$ to $V_3$. Now $V_2-\{v_0,u\}+\{v_1,z\}$ remains accessible as both $v_0$ and $u$ are not movable, but now in addition $V_3-v_1+u$ becomes accessible with witness $v_2$, a contradiction to maximality of $a$.

\medskip
{\bf Case 2.3:} $|\mathcal{U}|=6$. As $q(v_0)\ge 6$, by Claim~\ref{q'}, there are $z\in Q'(v_0)$ with color class $W(z)$.
By the case $W(z)$ is reachable from $\FF_0(v_0)$, but this  contradicts Lemma~\ref{nB2}(b).

\subsection{Proof of the case $a=3$}

\subsubsection{Setup}\label{DP1}

Let ${\cal{A}}=\{V_1,V_2,V_3\}$,  ${{\cal{B}}}=\{W_1,\ldots, W_5\}$. 

We first  show  that  $V_2$ and $V_3$ can be chosen to be terminal classes.
Assume not, say  $V_2$ blocks $V_3$. 
Then there is a vertex $v_2\in V_2$ movable to $V_1$ and a vertex $v_3\in V_3$ movable to $V_2$.
We move $v_2$ to $V_1$, so $V_2$ becomes the smaller class. 
Notice that $v_2$ is movable from $V_1+v_2$ to $V_2-v_2$ and $v_3$ remains movable from $V_3$ to $V_2-v_2$. So in the new $\HH$, both $V_1+v_2$ and $V_3$ are terminal and $|V_2-v_2|=s-1$ while $|V_1+v_2|=s$. Thus, we can assume that both $V_2$ and $V_3$ are terminal classes.

\begin{lem}\label{lm5}
    For $2\leq j\leq 3$, each solo vertex $v\in V_j$  has neighbors in  $V_1$ and $V_{5-j}$, and thus is ordinary.
\end{lem}

\begin{proof}
    Since both $V_2$ and $V_3$ are terminal classes, without loss of generality we can assume that there is a movable $v\in V_2$ that has a solo neighbor $u\in W(u)\in\BB$,
     and $v'\in V_3$ witnesses the directed edge $V_3V_1$ in $\HH$.

    If $v$ is movable to $V_3$, then we move $v'$ to $V_1$ and $v$ to $V_3$. In the new coloring $\varphi'$,
     $V_2-v$ as the smaller class,  $V_3-v'+v$ and $W(u)$ are accessible with regard to $V_2-v$. No other class in $\BB$ is accessible otherwise we get a larger $a$. $V_1+v'$ should not be in a strong component with classes other in $\HH$ since $V_1$ is not. However, if $V_1+v'$ can reach $V_2-v$, we also get a larger $a$. Thus $V_1+v'$ must be in a strong component by itself in the auxiliary digraph regarding the new coloring, but this contradicts Claim~\ref{44} as we would have no strong component of size $5$ and one strong component of size $1$.

    If $v$ is movable to $V_1$, then we  move  $v$ to $V_1$. Now we take $V_2-v$ as the smaller class, then $V_1$ and $W(u)$ are accessible with regard to $V_2-v$. Again, no other class in $\BB$ is accessible otherwise we get a larger $a$. $V_3$ would be in a strong component in the auxiliary digraph regarding the new coloring, but this contradicts Claim~\ref{44} as we would have no strong component of size $5$ and one strong component of size $1$.
\end{proof}

Denote the size of a largest strong component of $\HH$ contained in $\BB$   by $b_0$. By Claim~\ref{44},  either $b_0=3$ or $b_0=5$.

\medskip

\textbf{Case 3.1:} $b_0=3$.  By Claim~\ref{44}, we may assume that the vertex sets of strong components of $\HH$ contained in $\BB$ are $\BB_1=\{W_1,W_2\}$ and $\BB_2=\{W_3,W_4,W_5\}$. Recall that $V_0$ denotes the set of isolated vertices in $G$, and $n_0=|V_0|$.
 By the definition of $B$, $V_0\subset A$. Let $n'_0=|V_0-V_1|$.

Consider the following discharging procedure DP.

At the beginning, each edge of $G-x$ has charge $1$, so the sum of all charges is $|E(G-x)|$. Then
each edge $e=uv\in E(G-x)$ shares its charge among its ends according to the rules below.
    
 (R1)     if $v\in V_1$, then the edge sends all charge to $u$;

 (R2)      if $v\in A-V_1$ and $u$ is its solo neighbor in $B$, then the $e$ sends all charge to $u$;

 (R3)  in all other cases, $e$ sends $\frac{1}{2}$ to each endpoint.

So, denoting the charge of a vertex $v\in V(G)$ by $ch(v)$, we have 
\begin{equation}\label{ch}
    \sum\nolimits_{v\in V(G)}ch(v)=|E(G-x)|.
\end{equation}

\medskip
If a non-isolated vertex $v\in A-V_1$  has a solo neighbor in $B$, then by Lemma~\ref{lm5}
it has a neighbor in each of the other two classes in $\AA$, thus by rules (R2) and (R3) its charge  is  at least $\frac{1}{2}+1=\frac{3}{2}$. 
If this non-isolated $v\in A-V_1$ has no solo neighbors, then again by (R3) or (R2), $v$  receives  charge  at least $\frac{1}{2}$ from each incident edge, and hence $ch(v)\geq \frac{3}{2}$.

Each vertex $u\in B$  receives at least $3$ from  the edges connecting $u$ with $A$.
Since $\BB_1$ and $\BB_2$ are vertex sets of disjoint strong components of $\HH$, at least $s$ edges 
connect any class in $\BB_1$ with any class in $\BB_2$. Hence the vertices of $B$ receive  total charge at least $6s$ from these edges. Thus,
$$\sum\nolimits_{v\in V(G)}ch(v)\geq (2s-1-n'_0)\cdot\frac{3}{2}+5s\cdot 3+6s=24s-n'_0-\frac{3}{2}>3(8s-1-n'_0)-6.$$
Together with~\eqref{ch}, this contradicts Lemma~\ref{pla}(b).
\medskip

\textbf{Case 3.2: }$b_0=5$. For each $u\in B$ and $v\in V_2$,
define the weight $w_3(v,u)=\frac{1}{\|V,u\|}$. Then for each $v\in V_2$, define
$$w_3(v)=\sum\nolimits_{uv\in E(G): u\in B}w_3(v,u). $$
By definition, $\sum_{v\in V_2}w_3(v)=\sum_{v\in V_2,u\in B}w_3(v,u)=5s$.

Since $V_2$ is accessible, there is some $v\in V_2$ movable to $V_1$. 
Then by Lemma~\ref{lm5}, $v$ has no solo neighbor, so $w_3(v)\le 8\cdot\frac{1}{2}=4$. 
Thus there is some $v'\in V_2$ with $w_3(v')>5$.

Now we know that $v'$ has a neighbor in $V_1$ and a neighbor in $V_3$, so $|N(v')\cap B|\le 6$.
In order to achieve $w_3(v')>5$, we need $q(v')\ge 5$, and hence by Claim~\ref{q'}, $q'(v')\ge 4$.  By Lemma~\ref{nB2}(a), each neighbor of $v'$ in $Q'(v')$ must be in a class containing some other neighbor of $v'$, so there is some class $W'\in\BB$ that is not adjacent to $v'$. 
Then we pick some $z\in Q'(v')$ with color class $W(z)$.
By the case, $W(z)$ is reachable from $W'$, but by Lemma~\ref{lm5}, $v'$ is ordinary, and this leads to a contradiction to Lemma~\ref{nB2}(b).

\subsection{Proof of the case $a=5$}

In this case, since $d(x)\leq 5$ and $x$ has a neighbor in each class of $\AA$, we have $d(x)=5$ and $x$
  has no neighbors in $B$.
First,
we take a closer look at  $\HH[\AA]$. 

We call $\HH[\AA]$ \emph{nice}, if every accessible class other than $V_1$ blocks at most one accessible class. All $5$-vertex nice in-trees rooted at $V_1$ are listed in Figure~\ref{fig:nice}.
The two $5$-vertex in-trees rooted at $V_1$ with $d^-_{\HH[\AA]}(V_1)\geq 2$ that are not nice  are listed in Figure~\ref{fig:notnice}.

\begin{figure}
 \centering
 \includegraphics[scale=0.8]{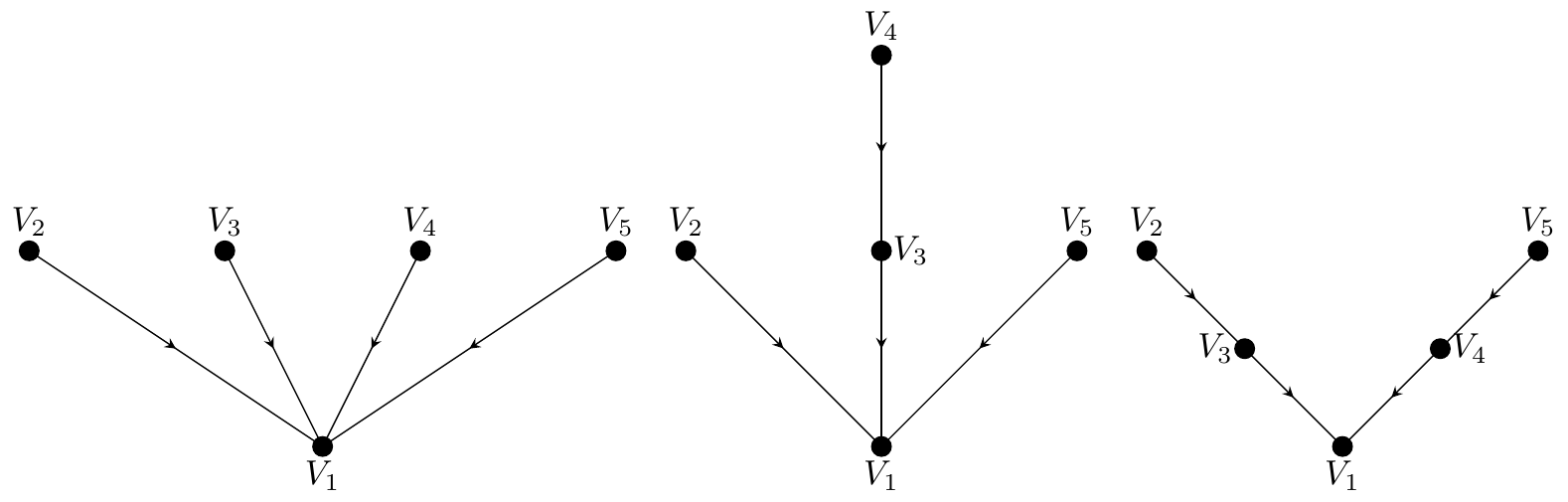}
 \caption{Nice digraphs: $\protect\overrightarrow{K_{4,1}}, \,\protect\overrightarrow{T_3}$ and $\protect\overrightarrow{T_{2,2}}$.}
 \label{fig:nice}
 \end{figure}

\begin{lem}\label{nice}
    If $a=5$, then we can choose an almost equitable coloring $\varphi$ so that $\HH[\AA]$ is nice.
\end{lem}

\begin{proof}
    Note that if $d^-_{\HH[\AA]}(V_1)\ge 3$, then  $\HH[\AA]$ is nice. So, we have the following cases.

    {\bf Case 1:} $d^-_{\HH[\AA]}(V_1)=2$. Then $\HH[\AA]$ contains one of the two digraphs in Figure~\ref{fig:notnice}.

\begin{figure}
 \centering
 \includegraphics{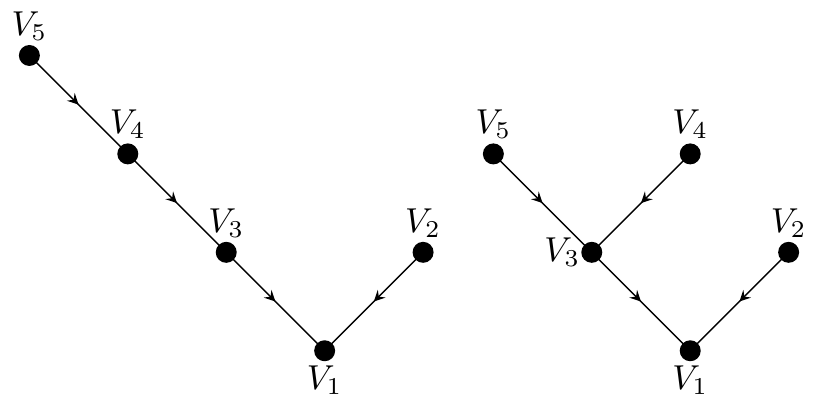}
 \caption{Digraphs with $d^-_{\HH[\AA]}(V_1)\ge 2$ that are not nice: $\protect\overrightarrow{T_{3,1}}$ and $\protect\overrightarrow{T'_{3,1}}$.}
 \label{fig:notnice}
 \end{figure}

  {\em Case 1.1:} $\HH[\AA]$ contains $\protect\overrightarrow{T'_{3,1}}$. 
   Let $\varphi'$ be obtained from $\varphi$ by moving  a witness $v_3$ of the arc $V_3V_1$ into $V_1$. 
  Then $V_3-v_3$ is the new small class, and the arcs $V_4(V_3-v_3)$, $V_5(V_3-v_3)$ and $(V_1+v_3)(V_3-v_3)$ are present 
   in the new $\HH$. So, if at least one of $V_1+v_3,V_3-v_3,V_4,V_5$ is an out-neighbor of $V_2$ in the new $\HH$, then
the new $\HH[\AA]$ is nice.   Otherwise, $|E(V_2,V_1\cup V_3\cup V_4\cup V_5)|\geq 4s$, and hence
$$|E(V_2\cup V_6\cup V_7\cup V_8,V_1\cup V_3\cup V_4\cup V_5)|\geq 4s+ 4|V_6\cup V_7\cup V_8|\geq 16s= 2n,
$$
contradicting Lemma~\ref{pla}(b).

 {\em Case 1.2:} $\HH[\AA]$ contains $\protect\overrightarrow{T_{3,1}}$. If  $V_5V_2\in E(\HH)$, then $\HH[\AA]$ is nice, a contradiction. 
 So, $|E(V_2, V_5)|\geq s$.
 Again, let $\varphi'$ be obtained from $\varphi$ by moving a witness $v_3$
 of $V_3V_1$ into $V_1$. Again, $V_3-v_3$ is the new small class, and the arcs $V_4(V_3-v_3)$, $V_5(V_3-v_3)$ and $(V_1+v_3)(V_3-v_3)$ are present 
   in the new $\HH$. So, if  one of $V_1+v_3,V_3-v_3$ is an out-neighbor of $V_2$ in the new $\HH$, then
$\HH[\AA]$ is nice, and if $V_2V_4\in E(\HH)$, then we get Case 1.1. Otherwise, as in Case 1.1, 
$$|E(V_2\cup V_6\cup V_7\cup V_8,V_1\cup V_3\cup V_4\cup V_5)|\geq s+3s+ 4|V_6\cup V_7\cup V_8|\geq 16s= 2n,
$$
contradicting Lemma~\ref{pla}(b).
    
  {\bf Case 2:} $d^-_{\HH[\AA]}(V_1)=1$. Suppose $V_2V_1\in E(\HH)$ and $v_2\in V_2$ a witness of this arc.
  Since each vertex in $\AA$ is accessible, $d^-_{\HH[\AA]}(V_2)\geq 1$, say $V_3V_2\in E(\HH)$.
  Let $\varphi'$ be obtained from $\varphi$ by moving $v_2$ into $V_1$. Then $V_2-v_2$ is the new small class, all classes in
  $\AA$ are still accessible, and $V_2-v_2$ has at least two in-neighbors in the new $\HH$. So either the new $\HH$ is nice or
  we have Case 1.
\end{proof}


\begin{lem}\label{unmovable}
    If $\HH[\AA]$ is nice, then each solo vertex $v\in V_i\in \AA-V_1$ has a neighbor in each
    class of $\AA-V_i$. In particular, $v$
    is ordinary.
\end{lem}

\begin{proof} Suppose 
     $v\in V_i\in \AA-V_1$  has a solo neighbor $u\in W\in \BB$ and has no neighbor in $V_j$ for some $V_j\in \AA-V_i$.
   If $\HH-V_i$ has
    a $V_j,V_1$-path $P$, say  $P=W_1,W_2,\dots,W_\ell$, where $W_1=V_j$, $W_\ell=V_1$
   and $w_h$ is a witness of $W_hW_{h+1}$ for $h=1,\ldots,\ell-1$, then we change $\varphi$ as follows. 
   Since $x$ has no neighbors in $B$, move it into the class of $u$, then move $u$ to $V_i$, $v$ to $V_j=W_1$, 
   and then
   for $h=1,2,\dots,\ell-1$, move $w_h$ from $W_h$ to $W_{h+1}$. 
   This would yield an equitable coloring on $G$, so assume that $\HH-V_i$ has no such path.

    This means that $V_i$ blocks $V_j$. Since $\HH[\AA]$ is nice, $V_j$ is the unique vertex in $\HH[\AA]$
    blocked by $V_i$, and $v$ has neighbors in each  class of $\AA-V_j-V_i$.
    Since $V_i$ is the only out-neighbor of $V_j$ in $\HH[\AA]$, we have $|E_G(V_j,A-V_i)|\geq 3s$.
    
    
    If $u$ is not adjacent to some vertex $v'$ that is movable from $V_j$ to $V_i$, then we can move $v$ to $V_j$ and  $v'$
    to $V_i$. Since $\HH[\AA]$ is nice, all classes of $\AA$ remain accessible, but now the class of $u$ also becomes accessible,
    contradicting the maximality of $a$.
    Thus $u$ is adjacent to all vertices movable from $V_j$ to $V_i$. Let $M$ be the set of these movable vertices and $m=|M|$.

    Now we count the edges connecting $A\setminus V_j-v+u$ and $B\cup V_k+v-u$. 
    Since $v$ is adjacent to each class in $\AA-V_j-V_i$ and to $u$,  at most $4$ edges connect $v$ to $B-u$. 
    No vertex in $B-u$ is movable to $A-V_j$, thus 
    \begin{equation}\label{lll}
        |E_G(B+v-u,A-V_j-v+u)|\ge 4(3s-1)-4+3+1=12s-4.
    \end{equation}
    Since $|E_G(V_j,A-V_i-V_j)|\geq 3s$, we get
    \begin{equation}\label{ll}
      |E_G(V_j, A-V_j-v+u)|=|E_G(V_j,A-V_i-V_j)|+|E_G(V_j,V_i+u)|
      \ge 3s+s-m+m=4s.  
    \end{equation}   
    Summing~\eqref{lll} with~\eqref{ll} gives $16s-4$ edges in a bipartite planar graph with $8s-1$ vertices, a contradiction to Lemma~\ref{pla}(b).
\end{proof}

Suppose now that $\varphi$ satisfies
 Lemma~\ref{nice}. Recall that $V_0$ denotes the set of isolated vertices in $G$, and $n_0=|V_0|$.
 By the definition of $B$, $V_0\subset A$. Let $n'_0=|V_0-V_1|$.
 Consider the  discharging procedure DP described in Case 3.1 of Subsection~\ref{DP1}.
We  will show that the new charges of vertices of $G$ satisfy
\begin{equation}\label{charge}
\mbox{\em $ch(u)\geq 5$ for each $u\in B$, and $ch(v)\geq 2.5$ for each $v\in A-V_1-V_0$,   }    
\end{equation}
which would imply that 
$$E(G-x)=\sum_{w\in V(G)-V_1-V_0}ch(w)\geq 5(3s)+2.5(4s-n'_0)=25s-2.5n'_0>3(|V(G)|-n_0).$$ 
Together with~\eqref{ch}, this contradicts Lemma~\ref{pla}(a).
Thus, it remains to prove~\eqref{charge}.

For   $u\in B$ and $V_i\in \AA$, $u$ has a neighbor in $V_i$. If it is a unique neighbor of $u$ in $V_i$,
then $u$ gets $1$ from $uv$ by (R2), otherwise  at least two edges connect $u$ to $V_i$ and $u$ gets $1/2$ from each of them. This proves the first part of~\eqref{charge}.

If $v\in A-V_1$  has a solo neighbor in $B$, then by Lemma~\ref{unmovable}, it has an edge to $V_1$ (from which it gets $1$ by (R1)) and at least $3$ edges to other classes in $\AA$ (from each of which it gets $1/2$ by (R3)). Thus in this case the second part of~\eqref{charge} holds.

Finally, if $v\in A-V_1-V_0$  has no solo neighbors in $B$, then  $v$ receives
by (R3) a charge of  $\frac{1}{2}$ from each incident edge, and by the case, there are at least $5$ of them. This 
proves~\eqref{charge} and hence finishes the proof of  Theorem~\ref{pla8}.
\qed

{\bf Acknowledgment.} We thank Dan Cranston, Hal Kierstead, Kittikorn Nakprasit and an anonymous referee for helpful comments.

\end{document}